\documentclass{article}

  \PassOptionsToPackage{numbers, compress}{natbib}

\usepackage[preprint]{neurips_2019}

\usepackage[utf8]{inputenc} 
\usepackage{lmodern}
\usepackage[T1]{fontenc}    
\usepackage{hyperref}       
\usepackage{url}            
\usepackage{amsfonts}       
\usepackage{nicefrac}       
\usepackage{microtype}      

\usepackage{amsmath,amssymb,amsthm}
\usepackage[]{algorithm,algorithmic} 
\usepackage{mathtools}
\usepackage{tcolorbox}
\usepackage{todonotes}
\usepackage{tensor}
\usepackage{caption}

\newcommand{\RR}{\mathbb{R}}
\newcommand{\symMat}{\mathbb{S}}
\newcommand{\EE}{\mathbb{E}}

\usepackage{accents}

\mathchardef\mhyphen="2D 
\DeclareMathOperator*{\argmin}{\arg\!\min}

\DeclarePairedDelimiter{\norm}{\lVert}{\rVert}
\DeclarePairedDelimiter{\abs}{\lvert}{\rvert}

\DeclareMathOperator{\rank}{rank}

\newtheorem{theorem}{Theorem}
\numberwithin{theorem}{section}

\newtheorem{lemma}[theorem]{Lemma}

\newcommand{\xsol}{X_\star}
\newcommand{\ysol}{y_\star}

\newcommand{\rsol}{r_\star}
\newcommand{\pval}{p_\star}
\newcommand{\dval}{d_\star}
\newcommand{\Amap}{\mathcal{A}}
\newcommand{\Ajmap}{\Amap^*}
\newcommand{\xsolset}{\mathcal{X}_\star}
\newcommand{\ysolset}{\mathcal{Y}_\star}
\newcommand{\dm}{n}
\newcommand{\ncons}{m}

\newcommand{\inprod}[2]{\langle #1, #2 \rangle}
\newcommand{\twonorm}[1]{\left\|#1\right\|_2}
\newcommand{\fronorm}[1]{\left\|#1\right\|_{\mbox{\tiny{F}}}}

\DeclareMathOperator{\dist}{dist_{\mbox{\tiny{F}}}}
\DeclareMathOperator{\tr}{trace}

\newcommand{\skcons}{3\sqrt{2}}

\let\originalleft\left
\let\originalright\right
\renewcommand{\left}{\mathopen{}\mathclose\bgroup\originalleft}
\renewcommand{\right}{\aftergroup\egroup\originalright}

\title{Bundle Method Sketching for Low Rank Semidefinite Programming}

\author{%
	Lijun Ding \\
	Cornell University\\
	Ithaca, NY 14850, USA \\
	\texttt{ld446@cornell.edu} \\
	\And
	Benjamin Grimmer \\
	Cornell University\\
	Ithaca, NY 14850, USA \\
	\texttt{bdg79@cornell.edu} \\
}

\begin{document}
	
	\maketitle
	
	\begin{abstract}
		In this paper, we show that the bundle method can be applied to solve semidefinite programming problems with a low rank solution without ever constructing a full matrix. To accomplish this, we use recent results from randomly sketching matrix optimization problems and from the analysis of bundle methods. Under strong duality and strict complementarity of SDP, our algorithm produces primal and the dual sequences converging in feasibility at a rate of $\tilde{O}(1/\epsilon)$ and in optimality at a rate of $\tilde{O}(1/\epsilon^2)$. Moreover, our algorithm outputs a low rank representation of its approximate solution with distance to the  optimal solution at most $O(\sqrt{\epsilon})$ within $\tilde{O}(1/\epsilon^2)$ iterations.
	\end{abstract}
	
	\section{Introduction}
We consider solving semidefinite problems of the following form:
\begin{equation}\label{p} \tag{P}
\begin{aligned}
& \underset{X\in\symMat^{\dm}\subset \RR^{\dm \times \dm}}{\text{maximize}}
& & \langle -C, X \rangle \\
& \text{subject to}
& & \Amap X = b \\
&&& X \succeq 0,
\end{aligned}
\end{equation}
where $C\in \symMat^{\dm}\subset \RR^{\dm \times \dm}$, $\Amap: \symMat^{\dm} \rightarrow \RR^{\ncons}$ being linear, and $b \in \RR^{\dm}$.
Denote the solution set as $\xsolset$. 
To accomplish the task of solving \eqref{p}, we consider the dual problem:
\begin{equation}\label{d} \tag{D}
\begin{aligned} 
& \underset{y\in\RR^{\ncons}}{\text{minimize}}
& & \langle -b, y \rangle \\
& \text{subject to}
& & \Amap^*y \preceq C,
\end{aligned}
\end{equation}
whose solution set is denoted as $\ysolset$. 
Then for all sufficiently large $\alpha$, e.g., larger than the trace of any solution $\xsol\in \xsolset$ \cite[Lemma 6.1]{ding2019optimal}, we can reformulate this as
\begin{equation}
\begin{aligned}\label{eq: penaltySDP}
& \underset{y\in\RR^{\ncons}}{\text{minimize}}
& & F(y) = \langle -b, y \rangle - \alpha \min\{\lambda_{\min}(C - \Ajmap y), 0\}.
\end{aligned}
\end{equation}

We propose applying the bundle method to solve this problem, which generates a sequence of dual solutions $y_t$. While the bundle method runs on the dual problem, a primal solution $X_t$ can be constructed through a series of rank one updates corresponding to subgradients of $F$. 
However maintaining such a primal solution greatly increases memory costs.
Fortunately, the primal problem enjoys having a low 
rank solution in many applications, e.g., matrix completion \cite{srebro2005rank} and phase retrieval \cite{candes2013phaselift}. Also, without specifying the 
detailed structure the problem, there always exists a solution $\xsol$ to \eqref{p} with rank $\rsol$ satisfying $\frac{\rsol(\rsol+1)}{2}\leq \ncons$ \cite{pataki1998rank}

To utilize the existence of such a low rank solution, we employ the matrix sketching methods introduced in \cite{tropp2017practical} and specifically for optimization algorithm in \cite{yurtsever2017sketchy}. The main idea is the following:
the skecthing method forms a linear sketch of the column and row spaces of the primal decision variable $X$, and
then uses the sketched column and row spaces to recover the primal decision variable. The
recovered decision variable approximates the original $X$ well if $X$ is (approximately) low rank. 
Notably, we need not store the entire decision variable $X$ at each iteration, but only the
sketch. Hence the memory requirements of the algorithm are substantially reduced.

\paragraph{Our Contributions.}
Our proposed sketching bundle method produces a sequence of dual solutions $y_k$ and a sequence of primal solutions $X_t$ (which are sketched by low rank matrices $\hat{X}_t$).
This is done without ever needing to write down the solutions $X_t$, which can be a substantial boon to computational efficiency.

In particular, we consider problems satisfying the following pair of standard assumptions: (i)~{\bf Strong duality} holds, meaning that there is a solution pair $(\xsol,\ysol)\in \xsolset\times \ysolset$ satisfying
\[
\pval := \inprod{-C}{\xsol } = \inprod{-b}{\ysol}=:\dval
\]
and (ii)~{\bf Strict Complementarity} holds, meaning there is a solution pair $(\xsol,\ysol)$ satisfying
\[
\rank(\xsol)+\rank(C-\Ajmap(\ysol))=\dm. 
\]
Under these condtions, we show $X_t$ and $y_t$ converge in terms of primal and dual feasibility at a rate of $\widetilde O(1/t)$ and the optimality gap converges at a rate of $\widetilde O(1/\sqrt{t})$. In particular, all three of these quantities converge .
\begin{theorem}[Primal-Dual Convergence]\label{thm:rates}
	Suppose the sets $\xsolset,\ysolset$ are both compact, and that strong duality and a strict complementary condition holds. For any $\epsilon>0$, Algorithm~\ref{alg:aggregation} with properly chosen parameters produces a solution pair $X_t$ and $y_t$ with
	$$t \leq \widetilde O\left(1/\epsilon\right)$$
	that satisfies
	\begin{align*}
		\text{approximate primal feasibility: }& \quad\|b - \Amap X_{t}\|^2 \leq \epsilon, \quad  X_{t}\succeq 0,\\
		\text{approximate dual feasibility: }& \quad \lambda_{\min}(C - \Ajmap y_{t})\geq -\epsilon,\\
		\text{approximate primal-dual optimality: }&\quad  |\langle b, y_t \rangle - \langle C, X_t\rangle| \leq \sqrt{\epsilon}
	\end{align*}
\end{theorem}

Moreover, we show that assuming all of the minimizers in $\xsolset$ are low rank, the sketched primal solutions $\hat{X}_t$ converge to the set of minimizers at the following rate.
\begin{theorem}[Sketched Solution Convergence] \label{thm:sketch}
	Suppose the sets $\xsolset,\ysolset$ are both compact, strong duality and a strict complementary condition holds, and all solutions $\xsol\in \xsolset$ have rank at most $r$. For any $\epsilon>0$, Algorithm~\ref{alg:aggregation} with properly chosen parameters produces a sketched primal solution $\hat{X}_t$ with
	$$t \leq \widetilde O(1/\epsilon^2)$$
	that satisfies $\EE \dist(\hat{X}_t,\xsolset) \leq \sqrt{\epsilon}.$
\end{theorem}
	\section{Defining The Sketching Bundle Method} \label{sec:def}
Our proposed proximal bundle method relies on an approximation of $F$ given by $\widetilde F^{t}(y) = \max\{F(z_t) + \langle g_t, y-z_t\rangle, \bar{F}^t(y)\}$ where $\bar{F}^t(y)$ is a convex combination of the lower bounds from previous subgradients. 
Notice that a subgradient $g_k$ of $F$ at $z_t$ can be computed as
\begin{align}
v_t &= \begin{cases} \mathtt{MinEigenvector}(C - \Ajmap z_t) & \quad \text{if } \lambda_{\min}(C-\Ajmap z_t)\leq 0 \label{eq:eigenvector-computation}\\
0 &\quad \text{otherwise} \end{cases}\\
g_k &= -b +\alpha \Amap v_tv_t^* \in \partial F(z_t).\label{eq:subgradient-computation}
\end{align}

Each iteration then computes a proximal step on this piecewise linear function given by
\begin{equation} \label{eq:subproblem}
	z_{t+1} \in \argmin_{y\in\RR^d} \widetilde F^t(y) + \frac{\rho}{2}\|y-y_t\|^2
\end{equation}
for some $\rho>0$. The optimality condition of this subproblem ensures that for some $\theta_t\in[0,1]$
$ 0=\theta_t g_t + (1-\theta_t)\left(\nabla\bar{F}^t(z_t) + \rho(z_{t+1}-y_t)\right). $
Our aggregate lower bound $\bar{F}^t(y)$ is updated to match this certifying subgradient
$ \bar{F}^{t+1}(y) = \bar{F}^t(z_{t+1}) + \langle\theta_t g_t + (1-\theta_t)\nabla\bar{F}^t(z_t), y-z_t\rangle.$
Then the following is an exact solution for the subproblem~\eqref{eq:subproblem}:
\begin{align}
	\theta_{t} &= \min\left\{1, \frac{\rho(F(z_t)-\bar{F}^t(z_t))}{\|g_t-\nabla\bar{F}^t(z_t)\|^2}\right\}\label{eq:theta}\\
	z_{t+1} &= y_t - \frac{1}{\rho}\left(\theta_t g_t +(1-\theta_t)\nabla \bar{F}^t(z_t)\right).\label{eq:subproblem-solution}
\end{align}

If the decrease in value of $F$ from $y_t$ to $z_{t+1}$ is at least $\beta$ fraction of the decrease in value of $\widetilde F^t$ from $y_t$ to $z_{t+1}$, then the bundle method sets $y_{t+1} = z_{t+1}$ (called a descent step). Otherwise the method sets $y_{t+1} = y_t$ (called a null step).
\begin{algorithm} 
	\caption{Bundle Method with Cut Aggregation} \label{alg:aggregation}
	\begin{algorithmic}
		\STATE {\bf Input:}  $A\colon\RR^{n\times n}\rightarrow \RR^d$, $b\in\RR^d$, $C\in\RR^{n\times n}$ , $y_1\in\RR^d$, $\rho>0,\ \beta\in(0,1),\ T\geq0$, $\alpha>0$
		\STATE Define $\bar{F}^1(\cdot)=-\infty$, $z_1=y_1$, $X_1=0$
		\FOR{$t=1, 2,\dots T$}
		\STATE Compute $v_t$ and $g_t$ as in~\eqref{eq:eigenvector-computation} and~\eqref{eq:subgradient-computation} \hfill {\it Compute One Subgradient}
		\STATE Compute $\theta_t$ and $z_{t+1}$ as in~\eqref{eq:theta} and~\eqref{eq:subproblem-solution} \hfill {\it Compute One Step}
		\STATE {\bf if} $F(z_{t+1}) \leq F(y_t) - \beta\left(F(y_t) - \widetilde F^k(z_{t+1})\right)$ {\bf then} $y_{t+1} = z_{t+1}$ \hfill {\it Take Descent Step}
		\STATE {\bf else} $y_{t+1} = y_{t}$ {\bf end if} \hfill {\it Take Null Step}
		\STATE Set $\bar{F}^{t+1}(y) = \bar{F}^t(z_{t+1}) + \langle\theta_t g_t + (1-\theta_t)\nabla\bar{F}^t(z_t), y-z_t\rangle$ \hfill {\it Update Model of $F$}
		\ENDFOR
	\end{algorithmic}
\end{algorithm}

\paragraph{Extracting Primal Solutions Directly.}
A solution to the primal problem~\eqref{p} can be extracted from the sequence of subgradients $g_k$ (as these describe the dual of the dual problem). Set our initial primal solution to be $X_{0}=0$. When each iteration updates the model $\bar F^t$, we could update our primal variable similarly:
\begin{align}
X_{t+1} &= \theta_t\alpha v_{t}v_{t}^* + (1-\theta_t)X_t. \label{eq:direct-primal-recurrence}
\end{align}
As stated in Theorem~\ref{thm:rates}, these primal solutions $X_{t}$ converge to optimality and feasibility at a rate of $O(1/t)$.
Alas this approach requires us to compute and store the full $n\times n$ matrix $X_t$ at each iteration. Assuming every solution is low rank, an ideal method would only require $O(nr)$ memory. The following section shows matrix sketching can accomplish this.

\paragraph{Extracting Primal Solutions Via Matrix Sketching.}
Here we describe how the matrix sketching method of~\cite{tropp2017practical} can be used to store an approximation of our primal solution $X_t\in \symMat^{n}_+$.
First, we draw two matrices with independent standard normal  entries
\begin{equation}
\begin{aligned}
\Psi \in \RR^ {n \times k} \quad \text{with} \quad k=2r+1; \\
\Phi \in \RR ^{l\times \dm} \quad \text{with} \quad l=4r+3;
\end{aligned}\nonumber
\end{equation} 
Here $r$ is chosen by the user. It either represents the estimate of the true rank of the primal solution or the user's computational budget in dealing with larges matrices. 

We use $Y^C_t$ and $Y^R_t$ to capture the column space  and the row space  of $X_t$:
\begin{align}\label{eqn: onetimeSketch}
Y^C_t = X_t\Psi \in \RR^{\dm \times k},\qquad  Y^R_t =\Phi X_t \in \RR^{l\times n} .
\end{align}
Hence we initially have $Y^C_0=0$ and $Y^R_0=0$.
Notice Algorithm~\ref{alg:aggregation} does not observe the matrix $X_t$ directly. Rather, it observes a stream of rank one updates
\[ X_{t+1} = \theta_t\alpha v_{t}v_{t}^*+ (1-\theta_t)X_t.\] In this setting, $Y^C_{t+1}$ and $Y^R_{t+1}$ can be directly computed as
\begin{align}
Y^C_{t+1} = \theta_t\alpha v_{t}\left(v_{t}^*\Psi\right) + (1-\theta_t)Y^C_t \in \RR^{\dm \times k }, \label{eq:sketch-primal-update1}\\
Y^R_{t+1} = \theta_t\alpha\left(\Psi v_{t}\right)v_{t}^* + (1-\theta_t)Y^R_t\in \RR^{l\times n} . \label{eq:sketch-primal-update2}
\end{align}
This observation allows us to form the sketch $Y^C_t$ and $Y^R_t$ from the stream of updates.

We then reconstruct $X_t$ and get the reconstructed matrix $\hat{X}_t$ by
\begin{align}\label{eqn: reconstructionsketch}
Y^C_t = Q_tR_t, \quad  B_t = (\Phi Q_t)^{\dagger} Y^R_t, \quad \hat{X}_t= Q_t[B_t]_r,
\end{align}
where $Q_tR_t$ is the $QR$ factorization of $Y^C_t$ and $[\cdot]_r$ returns the best rank $r$ approximation in Frobenius norm.
Specifically, the best rank $r$ approximation of a matrix $Z$ is $U\Sigma V^*$,
where $U$ and $V$ are right and left singular vectors corresponding to the $r$
largest singular values of $Z$ and $\Sigma$ is a diagonal matrix with $r$ largest singular values of $Z$. In actual implementation, we 
may only produce the factors $(QU, \Sigma, V)$ defining $\hat{X}_T$ in the end instead reconstructing $\hat{X}_t$ in every iteration.

Hence the expensive primal update~\eqref{eq:direct-primal-recurrence} can be replaced by much more efficient operations~\eqref{eq:sketch-primal-update1} and~\eqref{eq:sketch-primal-update2}. Then a low rank approximation of $X_t$ can be computed by~\eqref{eqn: reconstructionsketch}.

We remark that 
the reconstructed matrix $\hat{X}_t$ is not necessarily positive definite. However, this suffices for the purpose of finding a matrices close to $X_t$. More sophisticated procedure is available for producing a positive semidefinite approximation of $X_t$ \cite[Section 7.3]{tropp2017practical}.

	\section{Numerical Experiments}

In this section, we demonstrate Algorithm \ref{alg:aggregation} equipped with the sketching procedure does solve problem instances in four important problem classes: (1)  Generalized eigenvalue \cite[Section 5.1]{boumal2018deterministic}, (2) $\mathbb{Z}_2$ synchronization \cite{bandeira2018random}, (3) Max-Cut \cite{goemans1995improved}, and (4) matrix completion\cite{srebro2005rank}. For all experiments, we set $\beta =\frac{1}{4}$, $\rho=1$, and $\alpha = 2\inprod{I}{\xsol}$ where $\xsol$ is computed via MOSEK \cite{mosek2010mosek} or prior knowledge of the problem. We present the results of convergences in Figure \ref{fig:totalfigure} for the following problem instances (results for other problems instances are similar for each problem class): Let $\mathcal{D}$ be the distribution of symmetric matrices in $\symMat^{800}$ with upper triangular part (including the diagonal) being independent standard Gaussians.
\begin{enumerate}
	\item Generalized eigenvalue (\textbf{GE}): $C\sim \mathcal{D}$, $\mathcal{A}(X) = \inprod{B}{X}$ for any $X$, where $B = I + \frac{1}{n}WW^*$ and $W\sim \mathcal{D}$ and is independent of $C$, and $b=1\in \RR$.
	\item $\mathbb{Z}_2$ synchronization (\textbf{Z2}): $C= J +\frac{1}{2\sqrt{5}}W$ where $J\in \symMat^{800}$ is the all one matrix and $W\sim  \mathcal{D}$, $\Amap = \mathbf{diag}$, and $b\in \RR^{800}$ is the all one vector. 
	\item Max-Cut(\textbf{MCut}): $C=-L$ where $L$ is the Laplacian matrix of the G1 graph \cite{Gset} with $800$ vertices, $\Amap$ and $b$ are the same as $\mathbb{Z}_2$ synchronization
	\item Matrix  Completion(\textbf{MComp}): A random rank $3$ matrix $X^\natural\in \symMat^{400}$ is generated. The index set $\Omega \subset  \{1,\dots,400\}\times \{1,\dots,400\}=:[400]^2$ is generated in a way that each  $(i,j)\in [400]^2$ is in $\Omega$ with probability $0.2$ independently from anything else. Set $C = I\in \symMat^{800}$. The linear constraint is $ X_{n+i,n+j} = X^\natural_{i,j} $ for each $(i,j)\in \Omega$. So $[\Amap(X)]_{i,j} =X_{i+n,j+n}$ and $b_{i,j} =X^\natural_{i,j}$ for each $(i,j)\in \Omega$.
\end{enumerate}
As can be seen from the experiments, with $2000$ iterations or less, the dual and primal objective converges fairly fast except the matrix completion problem. The infeasibility measured by $\norm{\Amap X-b}$ and distance to solution is about $10^{-2}$ for most of the problems except \textbf{GE}. In general, we note the  convergence is quicker when the problems have rank $1$ optimal solutions (\textbf{GE} and \textbf{Z2}) comparing to problems with higher rank optimal solutions (\textbf{MCut} and \textbf{MComp}).
\begin{figure}[ht]
	\centering
	\includegraphics[width=0.8\linewidth]{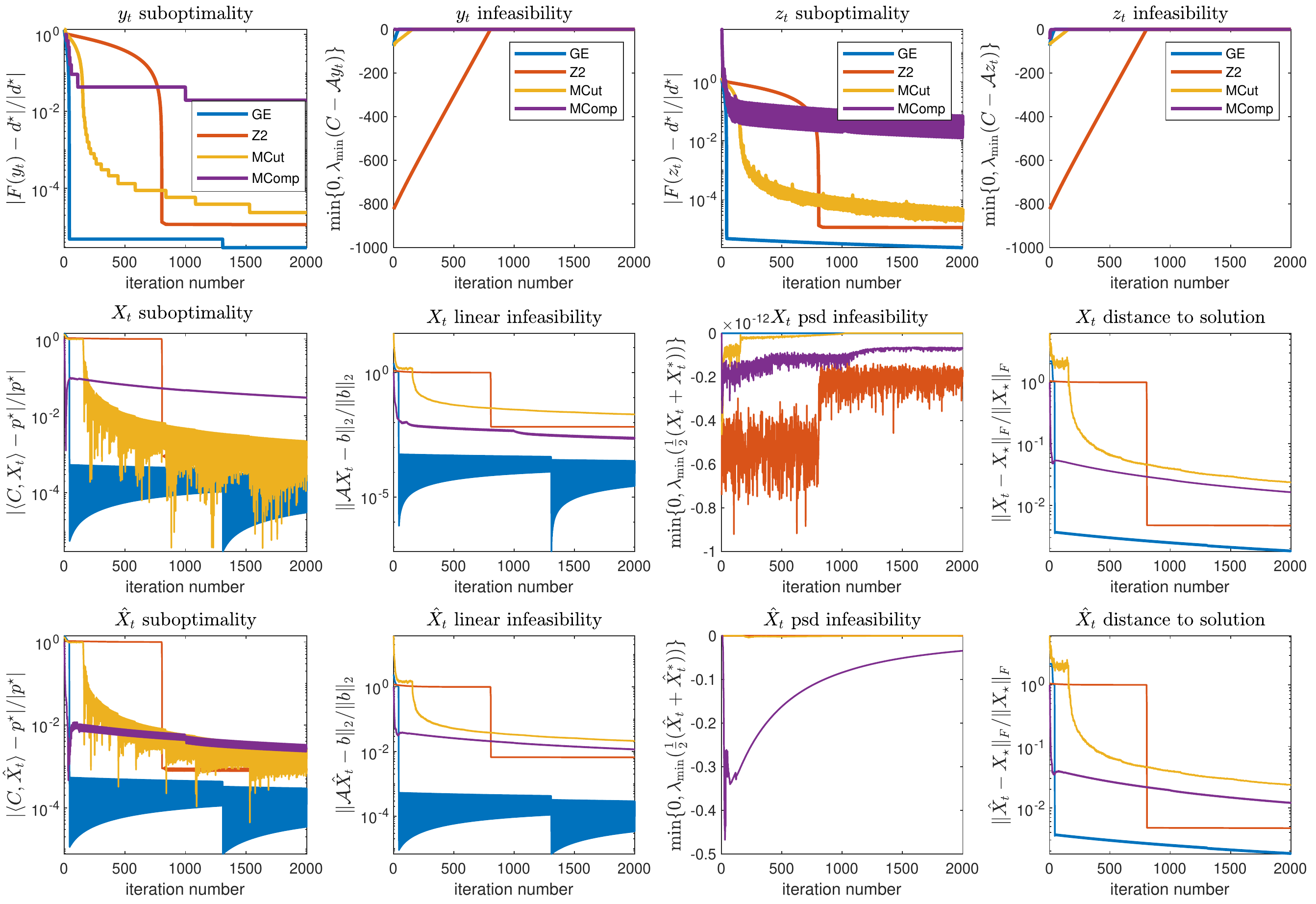}
	\caption{The upper, middle, and lower plots give the different convergence measures (as indicated on the title) for the dual iterates $z_t$ and $y_t$, the primal iterates $X_k$, and the sketching iterate $\hat{X}_t$ respectively. The color corresponds to different problem instances as indicated in the legend of the upper plots. Due to numerical error in updating $X_t$ and non-asymmetry of $\hat{X}_t$, we compute the minimum eigenvalue of their symmetrized version. $\lambda_{\min}(\frac{1}{2}(X_t+X_t'))$ appears to be negative but very small due to numerical error.}
	\label{fig:totalfigure}
\end{figure}

	
	\bibliographystyle{plain}
	{\small \bibliography{references}}
	
	\appendix
	\section{Proofs of Convergence Guarantees}

\subsection{Auxiliary lemmas}\label{sec: analtical conditon}

\begin{lemma}(Sketching Guarantee)\cite[Theorem 5.1]{tropp2017randomized} \label{lem: sketchguarantees}
	Fix a target rank $r$. Let $X$ be a matrix, and let $(Y,W)$ be a sketch of $X$ of the form~\eqref{eqn: onetimeSketch}. The procedure~\eqref{eqn: reconstructionsketch} yields a rank-$r$ matrix $\hat{X}$ with 
	\[
	\EE \fronorm{X-\hat{X}} \leq 3\sqrt{2} \fronorm{X-[X]_r}.
	\]
 Here $\fronorm{\cdot}$ is the Frobenius norm. Similar operator $2$-norm bounds hold with high probability. 
\end{lemma}

\begin{lemma}[Compact sublevel set]\label{lem: sublevelset}
	If a convex lower semicontinuous function $f(x): \RR^\dm \rightarrow \RR\cup \{\infty\}$ has a compact nonempty solution set, then all of its sublevel set is compact. 
\end{lemma}
\begin{proof}
	Suppose for some $L\in \RR$, then the closed sublevel set $S_L = \{x\mid f(x)\leq L\}$ is unbounded. Then there is a unit direction vector $\gamma \in \RR$ such that for all $x \in S_L,\alpha \geq 0$, $x+\alpha \gamma \in S_L$.  This in particular violates the fact the solution set is bounded and the proof is completed. 
\end{proof}
\begin{lemma}[Quadratic Growth]\cite[Section 4]{sturm2000error} \label{lem: qg}
If the solution sets $\xsolset$ and $\ysolset$ are compact and strict complementarity holds, then for any fixed $\epsilon>0$, there are some $\gamma_1,\gamma_2$ such that for all $y$ with $F(y)\leq \inprod{-b}{\ysol}+ \epsilon$ with $\alpha = 2 \sup_{\xsol\in\xsolset} \tr(\xsol)$ , and all $X\succeq 0$ with 
$|\inprod{C}{X}- \inprod{C}{\xsol}|\leq \epsilon$ and $\twonorm{\Amap{X}-b}\leq \epsilon$:
	\[
\dist^2(y,\ysolset) \leq  \gamma _1(F(y)-F(\ysol)),\quad \dist^2(X,\xsolset)\leq \gamma_2\left(\abs{\inprod{C}{X}-\inprod{C}{\xsol}}+\twonorm{\Amap{X}-b}\right).
\]
\end{lemma}
\begin{proof}
	Using the proof of Lemma \ref{lem:dual-feas}, we find that $\lambda_{\min}(C-\Amap^*(y)) \geq -\frac{\epsilon}{\sup_{\xsol \in \xsolset}\tr{\xsol}}$. Thus $\inprod{-b}{y}\leq \inprod{-b}{\ysol}+3\epsilon$. The result in \cite[Section 4]{sturm2000error} requires the set  $S_1=\{y\mid F(y)\leq \inprod{-b}{\ysol}+ \epsilon\}$, and the set  $S_2=\{X\mid X\succeq 0,
	|\inprod{C}{X}- \inprod{C}{\xsol}|\leq \epsilon,\,\text{and}\, \twonorm{\Amap{X}-b}\leq \epsilon\}$ being compact. 
	Using \cite[Theorem 7.21]{ruszczynski2006nonlinear}, the optimization problem 
	$\min_{X\succeq 0} g(X):= \inprod{C}{X} +\gamma \twonorm{\Amap X-b}$
	has the same solution set as 
	the primal SDP \eqref{p} for some large $\gamma >0$ . Thus the compactness of the set $S_1$ and $S_2$ is ensured by Lemma \ref{lem: sublevelset}, and the proof is completed.
\end{proof}

\subsection{Proof of Theorem~\ref{thm:rates}}
Let $D_{\xsolset} \geq \sup_{\xsol\in\xsolset} \tr(\xsol)$ and $D_y \geq \sup_{F(y)\leq F(y_0)} \|y\|_2.$ 
Then we set the bundle method's parameters to be $\alpha = 2 D_{\xsolset}$, $\beta=1/2$, and $\rho = 1/D_y^2$. We recall the inner product 
for matrices is the trace inner product and is the dot product for the vectors.

In the following three lemmas, we prove bounds on primal feasibility, dual feasibility, and optimality in terms of $F(y_{t}) - F(\ysol)$. From this, we can conclude these quantities converge at the claimed rate since the Du and Ruszcynskii~\cite{Du2017} recently showed the bundle method has $F(y_{t}) - F(\ysol)$ converge at a $\widetilde O(1/\epsilon)$ rate. 
\begin{lemma}[Primal Feasibility]\label{lem:primal-feas}
	At every descent step $t$, we have approximate primal feasibility
	$$ X_{t+1}\succeq 0, $$
	$$ \|b - \Amap X_{t+1}\|^2 \leq \frac{2(F(y_{t}) - F(\ysol))}{\beta D^2_y}. $$
\end{lemma}
\begin{proof}
	Noting that $X_{t+1}$ is built out of a convex combination of the rank matrices $v_tv_t^T$, its immediate that it is always a positive semidefinite matrix.
	
	The definition of $X_t$ immediately gives the following alternative characterization of $\bar F^{t+1}$,
	$$ \bar F^{t+1}(y) = -\langle C,X_{t+1}\rangle -\langle b - \Amap X_{t+1}, y\rangle.$$
	Since we constructed $\bar F^{t+1}$ to correspond to the first-order optimality condition of the subproblem~\eqref{eq:subproblem}, we have
	$$ 0 = \nabla \bar F^{t+1}(y_{t+1}) + \rho(z_{t+1} - y_t).$$ 
	Hence $\|b - \Amap X_{t+1}\|^2 = \rho^2\|y_{t+1} - y_t\|^2$. 
	The distance traveled during any descent step can be bounded by the objective value gap as
	$$ \frac{\rho}{2}\|y_{t+1} - y_t\|^2 \leq F(y_{t})-\widetilde F^{t}(y_{t+1}) \leq \frac{F(y_{t})-F(y_{t+1})}{\beta} \leq \frac{F(y_{t})-F(\ysol)}{\beta}$$
	where the first inequality uses the fact that $z_{t+1}$ minimizes $\widetilde F^{t}(\cdot)+\frac{\rho}{2}\|\cdot-y_t\|$ and the second inequality uses the definition of a descent step.
	Combining this with our feasibility bound shows
	$$ \|b - \Amap X_{t+1}\|^2\leq \frac{2\rho(F(y_{t}) - F(\ysol))}{\beta}. $$
	Then our choice of $\rho$ completes the proof.
\end{proof}

\begin{lemma}[Dual Feasibility]\label{lem:dual-feas}
	At every descent step $t$, we have approximate dual feasibility
	$$ \lambda_{\min}(C - \Ajmap y_{t+1})\geq \frac{-(F(y_{t}) - F(\ysol))}{D_{\xsolset}}.$$
\end{lemma}
\begin{proof}
	Standard strong duality and exact penalization arguments show for any $\xsol\in\xsolset$,
	\begin{align*}
	\langle b, y_{t+1}-\ysol \rangle & = \langle \Amap\xsol, y_{t+1} \rangle- \inprod{C}{\xsol}\\
	& \leq \langle \xsol, \Ajmap(y_{t+1}-C)\rangle\\
	& \leq -\tr(\xsol)\min\{\lambda_{\min}(C - \Ajmap y_{t+1}),0\}.
	\end{align*}
	Recalling our assumption that $\alpha \geq 2D \geq 2\tr(\xsol)$ yields the claimed feasibility bound
	\begin{align*}
	F(y_{t}) - F(\ysol) \geq F(y_{t+1}) - F(\ysol) & = \langle -b, y_{t+1}-\ysol \rangle -\alpha\min\{\lambda_{\min}(C - \Ajmap y_{t+1}),0\}\\
	& \geq -\tr(\xsol)\min\{\lambda_{\min}(C - \Ajmap y_{t+1}),0\}. \qedhere
	\end{align*}
\end{proof}

\begin{lemma}[Primal-Dual Optimality]\label{lem:optimality}
	At every descent step $t$, we have approximate primal-dual optimality bounded above by
	$$\langle b, y_{t+1} \rangle - \langle C, X_{t+1}\rangle \leq \frac{\alpha}{D_{\xsolset}}(F(y_{t}) - F(\ysol)) + \sqrt{\frac{2(F(y_{t}) - F(\ysol))}{\beta}}.$$
	and below by
	$$
	\langle b, y_{t+1} \rangle - \langle C, X_{t+1}\rangle \geq -\frac{1-\beta}{\beta}(F(y_{t}) - F(\ysol)) - \sqrt{\frac{2(F(y_{t}) - F(\ysol))}{\beta}}.$$
\end{lemma}
\begin{proof}
	The standard duality analysis shows the primal-dual objective gap equals
	\begin{align*}
	\langle b, y_{t+1} \rangle - \langle C, X_{t+1}\rangle &= \langle \Amap X_{t+1}, y_{t+1} \rangle - \langle C, X_{t+1}\rangle + \langle b-\Amap X_{t+1}, y_{t+1}\rangle\\
	&= \langle X_{t+1}, \Ajmap y_{t+1} - C\rangle + \langle b-\Amap X_{t+1}, y_{t+1}\rangle.
	\end{align*}
	Notice that the second term here is bounded above and below as
	$$ |\langle b-\Amap X_{t+1}, y_{t+1}\rangle| \leq \sqrt{\frac{2(F(y_{t}) - F(\ysol))}{\beta D_y^2}}\ \|y_{t+1}\|_2\leq \sqrt{\frac{2(F(y_{t}) - F(\ysol))}{\beta}}$$
	by Lemma~\ref{lem:primal-feas}. 
	Hence we only need to show that the first term also approaches zero (that is, we approach holding complementary slackness).
	
	An upper bound on this inner product follows from Lemma~\ref{lem:dual-feas} as
	$$ \langle X_{t+1}, \Ajmap y_{t+1} - C \rangle \leq -\tr(X_{t+1})\min\{\lambda_{\min}(C - \Ajmap y_{t+1}),0\} \leq \frac{\tr(X_{t+1})(F(y_{t})-F(\ysol))}{D_{\xsolset}}.$$
	Hence 
	$$\langle b, y_{t+1} \rangle - \langle C, X_{t+1}\rangle \leq \frac{\alpha}{D_{\xsolset}}(F(y_{t}) - F(\ysol)) + \sqrt{\frac{2(F(y_{t}) - F(\ysol))}{\beta}}.$$
	A lower bound on this inner product follows as
	\begin{align*}
		\frac{1-\beta}{\beta}(F(y_{t}) - F(y_{t+1})) &\geq F(y_{t+1}) - \widetilde F^{t}(y_{t+1})\\
		& = F(y_{t+1}) - \bar F^{t+1}(y_{t+1})\\
		& = -\alpha\min\{\lambda_{\min}(C-\Ajmap y_{t+1}),0\} + \langle C, X_{t+1}\rangle -\langle \Amap X_{t+1}, y_{t+1}\rangle\\
		& \geq \langle X_{t+1}, C - \Ajmap y_{t+1}\rangle,
	\end{align*}
	where the first inequality follows from the definition of a descent step. Hence
	\begin{align*}
		\langle b, y_{t+1} \rangle - \langle C, X_{t+1}\rangle &\geq -\frac{1-\beta}{\beta}(F(y_{t}) - F(\ysol)) -\sqrt{\frac{2(F(y_{t}) - F(\ysol))}{\beta}}.
	\end{align*}
\end{proof}

\subsection{Proof of Theorem~\ref{thm:sketch}}
	Using triangle inequality, we see that 
	\begin{equation}
	\begin{aligned}\label{eqn: converegncesketchfirst}
		\EE \dist(\hat{X}_t,\xsolset)\leq \EE\fronorm{\hat{X}_t-X_t} +\dist (X_t,\xsolset).
	\end{aligned}
    \end{equation} 
  The first term of \eqref{eqn: converegncesketchfirst} is bounded by 
  \[
  \EE\fronorm{\hat{X}_t -X_t} \overset{(a)}{\leq} \skcons \fronorm{X_t-[X_t]_r}\overset{(b)}{\leq} \skcons\dist(X_t,\xsolset),
  \]
   where step $(a)$ is due to Lemma \ref{lem: sketchguarantees}, and step $(b)$ is because $\xsol\in \xsolset$ all has rank less than or equal to $r$ and 
  $[X_t]_r$ is the best rank $r$ approximation of $X_t$ in terms of Frobenius norm. 
  
  Combining the above inequalities, we see that
  \begin{equation}
  \begin{aligned}\label{eqn: converegncesketchsecond}
  \EE \dist(\hat{X}_t,\xsolset)\leq (1+\skcons)\dist (X_t,\xsolset).
  \end{aligned}
  \end{equation}  
  Our task now is to have an estimate the rates convergence of $\dist(X_t,\xsolset)$. Denote the shorthand $g(X) =\abs{\inprod{C}{X}-\inprod{C}{\xsol}}+\gamma \twonorm{\Amap X -b}$. Using Lemma \ref{lem: qg}, there are some $\gamma$ such that for all $X$ with $g(X)\leq g(\xsol)+1$ satisfy 
  \begin{align}\label{eq: sketchyconvergencethird}
  \dist^2(X,\xsolset)\leq \gamma ( g(X)-g(\xsol)).
  \end{align}
  Then the theorem follows from Theorem \ref{thm:rates} as $g(X_t)-g(\xsol)$ converges at a rate of $\widetilde O(1/\epsilon^2)$.
  
  
  

\end{document}